\documentclass[12pt]{amsart}

\usepackage[style=alphabetic,backend=biber, sorting=nyt]{biblatex}
\bibliography{references}

\usepackage{amssymb}

\usepackage{amsmath}

\usepackage{enumitem}

\usepackage{nccmath}

\usepackage{tikz-cd}

\usepackage{amsthm}
\theoremstyle{definition} \newtheorem{theorem}{Theorem}[section]
\theoremstyle{definition} 
\theoremstyle{definition} \newtheorem{lemma}[theorem]{Lemma}
\theoremstyle{definition} 
\theoremstyle{definition} \newtheorem{conjecture}[theorem]{Conjecture}
\theoremstyle{definition} \newtheorem{corollary}[theorem]{Corollary}
\theoremstyle{remark} 
\theoremstyle{remark} \newtheorem{remark}[theorem]{Remark}
\theoremstyle{remark} 

\newcommand{\ssep}{\mid}

\newcommand{\ellphat}{\hat{\ell_p^{S}}}
\newcommand{\ellhat}{\hat{\ell_2^{S}}}
\newcommand{\ellpS}{\ell_p^{S}}
\newcommand{\ellS}{\ell_2^{S}}
\newcommand{\Sq}{\mathrm{Sq}}

\begin{document}

\title{Growth of homotopy groups of spheres via the Goodwillie-EHP Sequence}

\author{Guy Boyde}
\address{Mathematical Sciences, University of Southampton, Southampton SO17 1BJ, United Kingdom}
\email{guy.boyde@gmail.com}

\subjclass[2020]{Primary 55Q40; Secondary 18F50}
\keywords{Homotopy groups of spheres, EHP Sequence, Goodwille Tower}

\begin{abstract}
We bound the volume of the homotopy groups of the 2-local Goodwillie approximations of a sphere in terms of the amount of $2$-torsion in the stable stems, providing a Goodwillie-theoretic refinement of a result of Burklund and Senger. At the $2^k$-excisive approximation, this bound is obtained by `multiplying the stable answer by a polynomial of degree $k$'. The main tool is Behrens' Goodwillie-EHP Long Exact Sequence.
\end{abstract}

\maketitle

\section{Introduction}

Burklund and Senger \cite{BurklundSenger} have recently shown that the volume growth of the (stable and unstable) homotopy groups of spheres is subexponential. Their preprint appeared shortly after the first version of this one, and in light of their work this paper has been greatly revised. The object of this new version is to show that the Goodwillie Tower realises the `unstable part' of their bound as the limit of its Taylor series.

Let $p$ be prime. For an abelian group $A$, let $_{(p)}A$ be the \emph{$p$-torsion subgroup}: the subgroup consisting of elements of order a power of $p$. Let $$\ell_p(A):= \log_p(\mathrm{Card}(_{(p)}A)).$$ Write $\pi^S_*$ for the stable homotopy groups of spheres, and let $$\ellpS(q) := \ell_p(\pi_{q}^S) \textrm{, } \ellphat(q):= \max_{i \leq q} \ellpS(i).$$

The \emph{Goodwillie Tower of the identity}, due to Goodwillie \cite{GoodwillieI, GoodwillieII, GoodwillieIII} consists of the following data:

\begin{itemize}
    \item For each $k \in \mathbb{Z}_{\geq 0}$, a functor $P_k$ from based spaces to based spaces, called the \emph{$k$-excisive approximation to the identity}.
    \item For each $k \in \mathbb{Z}_{\geq 0}$, natural transformations $P_{k+1}(X) \to P_{k}(X)$, and $X \to P_k(X)$.
\end{itemize}

Letting $D_k(X)$ be the homotopy fibre of $P_k(X) \to P_{k-1}(X)$, we obtain a functor $D_k$, which is called the \emph{$k$-th layer} of the Goodwillie Tower. This data can be assembled into a commutative diagram as follows.

\begin{center}
\begin{tikzcd} & \vdots \ar[d] & \\
X \ar[r] \ar[dr] & P_k(X) \ar[d] & \ar[l] D_k(X) \\ 
  & P_{k-1}(X) \ar[d] & \ar[l] D_{k-1}(X) \\
  & \vdots \ar[d] & \\
  & P_1(X). &
\end{tikzcd}
\end{center}

Henceforth, we will work $2$-locally. Arone and Mahowald \cite{AroneMahowald} have shown that in the case of spheres, $D_k(S^n)$ is $2$-locally contractible unless $k$ is a power of $2$. It follows that $P_k(S^n)$ is 2-locally homotopy equivalent to $P_{2^j}(S^n)$, where $2^j$ is the largest power of $2$ which is at most $k$. It is therefore no loss to restrict attention to the layers $P_{2^k}(S^n)$ indexed by powers of $2$.

Our main result is the following bound on the size of the homotopy groups of the 2-local Goodwillie Tower on a sphere.

\begin{theorem} \label{mainThm} Let $n \geq 3$. For all $q \geq 0$ we have
$$\ell_2(\pi_{q+n}(P_{2^k}(S^n))) \leq \ellhat(q) \cdot \sum_{j=0}^k \frac{(q+1)^j}{j! \cdot 2^{\frac{1}{2}j(j-1)}}. $$
\end{theorem}

Informally, one might say that the size of the $2^k$-th Goodwillie filtration is controlled by the product of the stable size and a polynomial of degree $k$.

Mahler \cite{Mahler} notes that for $\alpha < 1$ the function $$F(q) := \sqrt{\frac{\log(\frac{1}{\alpha})}{2 \pi}} \int_{- \infty}^{+ \infty} e^{\frac{1}{2}\log(\alpha)x^2 + q \alpha^{xi - \frac{1}{2}}} dx$$ satisfies $F'(x)=F(\alpha x)$, hence has Taylor Series $$F(q) = \sum_{k=0}^\infty \alpha^{{\frac{1}{2}k(k-1)}}\frac{q^k}{k!}.$$ Taking $\alpha = \frac{1}{2}$ and letting $k \to \infty$, one immediately recovers (using convergence of the Goodwillie Tower as in Lemma \ref{connectivity}) the following result of Burklund and Senger, which also uses Mahler's paper.

\begin{corollary}[{\cite[Appendix A]{BurklundSenger}}] \label{kToInfty} Let $n \geq 3$. For all $q \geq 0$ we have $$
\ell_2(\pi_{q+n}(S^n)) \leq F(q+1) \cdot \ellhat(q).  $$
\end{corollary}

Our results are `relative to' the stable information encoded in $\ellS(q)$. Burklund and Senger show that $\ellS(q) = \exp(O(\log(q)^3))$, and conjecture that this is optimal in the sense that $\ellS(q) = \exp(\Theta(\log(q)^3))$.
Isaksen, Wang and Xu have made the following alternative conjecture.

\begin{conjecture}[\cite{IsaksenWangXu}] \label{IWXConjecture} There exists a nonzero constant $C$ such that $$\lim_{q \to \infty} \frac{\sum_{i=1}^q \ellS(i)}{q^2} = C.$$ \end{conjecture}

The results of this paper are more interesting if the second conjecture is true. Truth of Conjecture \ref{IWXConjecture} would imply that there existed constants $a$ and $b$ such that $\ellhat(q) \leq aq^2 +b$, hence by Theorem \ref{mainThm} that $$\ell_2(\pi_{t+n}(P_{2^k}(S^n))) \leq (aq^2+b) \sum_{j=0}^k \frac{(q+1)^j}{j! \cdot 2^{\frac{1}{2}j(j-1)}}.$$ In particular this would promote Theorem \ref{mainThm} to an absolute bound by a polynomial of degree $k+2$.

For the most part, this paper runs parallel to Appendix A of Burklund and Senger's paper \cite{BurklundSenger}, replacing the algebraic EHP sequence with Behrens' `Goodwillie-EHP' sequence. For this reason, our results are confined to spheres, and to $p=2$, but we expect that similar bounds will hold at odd primes. They prove their results for $p$-rank, and promote them to results about volume at the end, using the fact that spheres have finite homotopy exponent. We are able to work with volume directly, essentially because (Corollary \ref{LemmainVersion}) Hopf invariant is never divisible by more than a single power of 2. Burklund and Senger also state their result using big $O$ notation, stopping short of the completely explicit upper bound of Corollary \ref{kToInfty}. We suspect that this is because it makes no difference in their application, given the bounds they are able to prove on $\ellS(t)$. In particular, the concrete function in Corollary \ref{kToInfty} is not a result of using the calculus.

In the unpublished \cite{AroneKankaanrinta}, Arone and Kankaanrinta give an analogy between the Goodwillie Tower and the Taylor Series of the logarithm function, inverse to an analogy between stable homotopy and $e^{x-1}$. On this view one should think of the Goodwillie Tower as an infinite product, rather than an infinite sum, analogous to the following equation, which is obtained by exponentiating the Taylor Series of $\ln (1+(x-1))$.

$$ e^{x-1} \cdot e^{\frac{(x-1)^2}{2}} \cdot e^{\frac{(x-1)^3}{3}} \dots = x.$$

It may be interesting to consider Theorem \ref{mainThm} from this point of view.

\smallskip

I would like to thank Niall Taggart for his encouragement and helpful comments, and for making me aware of the paper \cite{AroneKankaanrinta} of Arone and Kankaanrinta. I would also like to thank Stephen Theriault and Charlotte Summers for their help and advice, and acknowledge the technical debt to Burklund and Senger. I am grateful for the support of an EPSRC Doctoral Prize.

\section{EHP Sequences and the Goodwillie Tower}

We will need the following lemma on connectivity of $P_{2^k}(S^n)$, which is noted by Johnson \cite{Johnson}. It follows immediately from the proof of Theorem 1.13 in \cite{GoodwillieIII}, using the fact that the identity functor is 1-analytic.

\begin{lemma} \label{connectivity} Let $n \geq 2$. The map $S^n \longrightarrow P_{2^k}(S^n)$ induces an isomorphism on $\pi_i$ for $i \leq (2^k+1)(n-1)$. \qed \end{lemma}

The classical 2-primary EHP sequence is due to James.

\begin{theorem}{\cite{James}} \label{JamesLemma} For $n \geq 1$, there is a fibre sequence
\[
\pushQED{\qed} 
\dots \xrightarrow{P} S^n \xrightarrow{E} \Omega S^{n+1} \xrightarrow{H} \Omega S^{2n+1} \xrightarrow{P} \dots . \qedhere
\popQED
\] \end{theorem}

Behrens \cite{Behrens} gives the following refinement, which incorporates the Goodwillie Tower.

\begin{theorem}{\cite[Corollary 2.1.4]{Behrens}} \label{BehrensLemma} For $k \geq 0$ and $n \geq 1$, there is a fibre sequence
\[
\pushQED{\qed} 
\dots \xrightarrow{P} P_{2^{k+1}}(S^n) \xrightarrow{E} \Omega P_{2^{k+1}}(S^{n+1}) \xrightarrow{H} \Omega P_{2^{k}}(S^{2n+1}) \xrightarrow{P} \dots . \qedhere
\popQED
\] \end{theorem}

The following corollary extracts the information that we will use.

\begin{corollary} \label{LemmainVersion} Let $k \geq 0$ and $n \geq 1$. If $i=2n-1$ then $$\ell_2(\pi_{i}(P_{2^{k+1}}(S^{n}))) \leq  \ell_2(\pi_{i+1}(P_{2^{k+1}}(S^{n+1}))) + 1,$$ and otherwise we have $$\ell_2(\pi_{i}(P_{2^{k+1}}(S^{n}))) \leq \ell_2(\pi_{i+1}(P_{2^{k+1}}(S^{n+1}))) + \ell_2(\pi_{i+2}(P_{2^k}(S^{2n+1}))).$$ \end{corollary}

The fact that we only have to add $1$ encodes the fact that Hopf invariant is never divisible by more than $2$.

The following lemma will be used to prove Corollary \ref{LemmainVersion}.

\begin{lemma} \label{firstCase} Let $A \xrightarrow{f} B \xrightarrow{g} C$ be an exact sequence of abelian groups, with $A$ torsion, and let $p$ be a prime. Then $\ell_p(B) \leq \ell_p(A) + \ell_p(C)$. \end{lemma}

\begin{proof} First, note that if an element of $B$ has order a power of $p$, then its image in $C$ also has order a power of $p$. Second, note that any element of ${_{(p)}}B$ which is hit by an element of $A$ must actually be hit by an element of ${_{(p)}}A$. To see this, suppose that some nonzero $y \in {_{(p)}}B$ is the image of $x \in A$. Since $A$ is torsion, $x$ must have finite order. The order of $y$ is $p^b$, for some $b \in \mathbb{N}$, so the order of $x$ is $u p^a$, for $a \geq b$ and $u$ coprime to $p$. Then $ux$ lies in ${_{(p)}}A$, and $f(ux)=uy$. Multiplication by $u$ is an isomorphism on $p$-torsion subgroups, so there is a unique element $z$ of ${_{(p)}}A$ with $uz = ux$, and we must have $f(z) = u^{-1}f(ux) = y$, as desired.

From these two observations we obtain an exact sequence $${_{(p)}}A \to {_{(p)}}B \to {_{(p)}}C.$$ It follows that $\mathrm{Card}({_{(p)}}B) \leq \mathrm{Card}({_{(p)}}A) \mathrm{Card}({_{(p)}}C)$, and the result follows by taking logarithms. \end{proof}

\begin{proof}[Proof of Corollary \ref{LemmainVersion}] The point is that the behaviour of Behrens' EHP sequence is sufficiently well governed by the classical one that, morally, one need only establish the result in that setting.

Consider the following portion of the long exact sequence on homotopy groups induced by the fibration of Theorem \ref{BehrensLemma}: $$\dots \rightarrow \pi_{i+2}(P_{2^k}(S^{2n+1})) \xrightarrow{P_*} \pi_{i}(P_{2^{k+1}}(S^{n})) \xrightarrow{E_*} \pi_{i+1}(P_{2^{k+1}}(S^{n+1})) \rightarrow \dots $$

By \cite[Proposition 3.1]{AroneMahowald}, since $S^{2n+1}$ is an odd sphere, $D_{2^k}(S^{2n+1})$ (the homotopy fibre of $P_{2^k}(S^{2n+1}) \to P_{2^{k-1}}(S^{2n+1})$) is rationally contractible for $2^k > 1$, i.e. for $k > 0$. This means that $\pi_{i+2}(P_{2^k}(S^{2n+1}))$ contains a class of infinite order if and only if $\pi_{i+2}(S^{2n+1})$ does, which happens if and only if $i+2 = 2n+1$ \cite{Serre}. By Lemma \ref{firstCase}, the result then follows for $i \neq 2n-1$.

We therefore restrict attention to the case $i=2n-1$. Consider the following commutative diagram, where the rows are the relevant portions of the long exact sequences on homotopy groups obtained from the EHP sequences, and the vertical maps are the natural transformations of the Goodwillie Tower.

\begin{tikzcd} \pi_{2n+1}(S^{n+1}) \ar[r, "H_*"] \ar[d] & \pi_{2n+1}(S^{2n+1}) \ar[r, "P_*"] \ar[d] & \pi_{2n-1}(S^{n}) \ar[d] \\
\pi_{2n+1}(P_{2^{k+1}}(S^{n+1})) \ar[r, "H_*"] & \pi_{2n+1}(P_{2^k}(S^{2n+1})) \ar[r, "P_*"] & \pi_{2n-1}(P_{2^{k+1}}(S^{n})). 
\end{tikzcd}

By Lemma \ref{connectivity}, the middle vertical map is an isomorphism, and therefore $\pi_{2n+1}(P_{2^k}(S^{2n+1})) \cong \mathbb{Z}$. We are interested in the contribution made by this copy of $\mathbb{Z}$ to the 2-torsion in $\pi_{2n-1}(P_{2^{k+1}}(S^{n}))$.

If $n$ is even, then we have seen that $\pi_{2n+1}(P_{2^{k+1}}(S^{n+1}))$ does not contain a class of infinite order, so $H_*$ is trivial. By exactness, $P_*$ is an injection. This means that no contribution is made to the torsion. More precisely, continuing the sequence to the right, the torsion subgroup of $\pi_{2n-1}(P_{2^{k+1}}(S^{n}))$ maps injectively into that of $\pi_{2n}(P_{2^{k+1}}(S^{n+1}))$, so $\ell_2(\pi_{i}(P_{2^{k+1}}(S^{n}))) \leq \ell_2(\pi_{i+1}(P_{2^{k+1}}(S^{n+1})))$, which implies the first case of the result.

If $n$ is odd, then the image of $H_*$ in the top row certainly contains twice the generator of $\pi_{2n+1}(S^{2n+1})$. Since the middle vertical is an isomorphism, this must also be true in the bottom row, so, again extending to the right, we obtain an exact sequence $$A \to \pi_{2n-1}(P_{2^{k+1}}(S^{n})) \to \pi_{2n}(P_{2^{k+1}}(S^{n+1})),$$ with $A$ equal to either $\mathbb{Z}/2$ or $0$. Applying Lemma \ref{firstCase} to this sequence then gives the first case of the result, as required. \end{proof}

\section{Representation by power series}

In this section, following Burklund and Senger, we reframe the discussion in terms of certain power series. As far as possible, our notation differs from theirs only in the addition of the Goodwillie parameter $k$. At this stage it is convenient to change from unstable to stable coordinates (replacing $\pi_i(S^n)$ with $\pi_{q+n}(S^n)$).

Let $$\mathcal{L}(k,n;t) := 1 + \sum_{q=1}^\infty \ell_2(\pi_{q+n}P_{2^{k}}S^{n}) \cdot t^q.$$

The addition of the constant term eliminates the case statement in Corollary \ref{LemmainVersion}, and we obtain the following.

\begin{lemma} \label{LemmaL} For all $k,n$ we have the coefficient-wise inequality of formal power series \[
\pushQED{\qed} 
\mathcal{L}(k, n ; t) \leq \mathcal{L}(k, n+1 ; t) + \mathcal{L}(k, 2n+1 ; t) \cdot t^{n-1}. \qedhere
\popQED
\]  \end{lemma}

For the corresponding stable object, let $\mathcal{L}^S(t)$ be the formal power series encoding the 2-local sizes $\ellS$. Formally: $$\mathcal{L}^S(t) = 1 + \sum_{q=1}^\infty \ellS(q) \cdot t^q.$$

Next, let $$I(k,n) = \{(i_1, \dots, i_k) \ssep i_k \geq n \textrm{ and } i_j \geq 2 i_{j+1}+1\}.$$ We include the empty sequence, which is regarded as the sole element of $I(0,n)$. In the language of \cite{Behrens}, this is the set of completely unadmissible sequences of excess $n$ and length $k$. We regard $I(k,n)$ as a graded set, by letting $$\dim(i_1, \dots, i_k) = \sum_j(i_j-1).$$ The empty sequence is regarded as having dimension $0$.
Let $A(k, n ; t)$ be the formal power series in $z$ defined by $$A(k, n ; t) = \sum_{q=0}^{\infty} \lvert \bigsqcup_{j \leq k} I(j,n)_q \rvert \cdot t^q,$$ and write $A(n;t)=A(\infty,n;t)$.

The following lemma reproduces Burklund and Senger's Lemma A.18 with the addition of the parameter $k$. The proof is identical.

\begin{lemma} \label{LemmaA} For all $k,n$ we have
\[
\pushQED{\qed} 
A(k+1, n ; t) = A(k+1, n+1 ; t) + A(k, 2n+1 ; t) \cdot t^{n-1}. \qedhere
\popQED
\] 
\end{lemma}

The point is that $A(k, n ; t)$ controls the relationship between the stable groups and the $2^k$-th Goodwillie filtration in the following sense.

\begin{lemma} \label{raisonDetre} We have the coefficient-wise inequality of formal power series $$ \mathcal{L}(k,n;t) \leq A(k,n ; t) \cdot \mathcal{L}^S(t).$$
\end{lemma}

\begin{proof} When $k=0$, the left hand side is just $\mathcal{L}^S(t)$, so the result follows since the empty sequence gives $A(k;n,t)$ constant term 1.

Now suppose $k \geq 1$. Assume the result holds for all $m > n$. This downwards induction is valid since for each fixed $q$, the coefficient of $t^q$ on the left hand side is zero for large $n$.

Then Lemmas \ref{LemmaA} and \ref{LemmaL}, together with the inductive hypothesis, give
\begin{align*} A(k,n ; t) \cdot \mathcal{L}^S(t) & = (A(k, n+1 ; t) + A(k, 2n+1 ; t) \cdot t^{n-1}) \cdot \mathcal{L}^S(t) \\
& = A(k, n+1 ; t) \cdot \mathcal{L}^S(t) + A(k, 2n+1 ; t) \cdot \mathcal{L}^S(t) \cdot t^{n-1} \\
& \geq \mathcal{L}(k, n+1 ; t) + \mathcal{L}(k, 2n+1 ; t) \cdot t^{n-1} \\
& \geq \mathcal{L}(k, n ; t),
\end{align*}
as required.
\end{proof}

\section{Bounding $A(k,n;t)$}

Burklund and Senger show \cite[Lemma A.20]{BurklundSenger} that the coefficients of $A(n;t)=A(\infty,n;t)$ are bounded above by counts of admissible sequences in the Steenrod algebra. They then use the fact that the dual Steenrod algebra is polynomial on generators in degrees roughly powers of $2$ to bound these counts, using work of Mahler \cite{Mahler}. Our goal in this section is to describe the corresponding situation for finite $k$ and recover the case $k = \infty$ as the limit.

Recall that the Steenrod Algebra $\mathcal{A}$ has a $\mathbb{Z}/2$-module basis consisting of \emph{admissible monomials} $\Sq^{i_1} \dots \Sq^{i_j}$ with $i_j \geq 2i_{j+1}$. Welcher \cite{Welcher} discusses a certain family of graded $\mathbb{Z}/2$-vector spaces $M_k$, which may be regarded as submodules of $\mathcal{A}$. Concretely, $M_k$ has a basis consisting of those admissible monomials $\Sq^{i_1} \dots \Sq^{i_j}$ with $j \geq k+1$ and $i_j \geq 2$.

For a graded $\mathbb{Z}/2$-vector space $V$, write $P(V;t)$ for the Poincar\'e series $\sum_{q=0}^\infty \dim_{\mathbb{Z}/2}(V_q) \cdot t^q$. The next lemma essentially just adds the parameter $k$ to Burklund and Senger's Lemma A.20. The proof is identical.

\begin{lemma} \label{toSteenrod} For $n \geq 3$, we have the coefficient-wise inequality
\[
\pushQED{\qed} 
A(k, n ; t) \leq 1 + P(M_0/M_k;t). \qedhere
\popQED
\] 
\end{lemma}

Welcher notes the formula for the Poincar\'e series of the successive quotients $M_{k-1}/M_k$ (in his statement there appears to be a typo in the indexing on the product). It can be seen by noticing that admissibility is equivalent to $(i_1,\dots,i_k)$ being of the form $$(2^{k}+2^{k-1}r_k+ \dots + 2r_2+r_1, \dots, 4+2r_k+r_{k-1},2+r_k)$$ for $r_1, \dots, r_k \geq 0.$ The degree of this expression is an affine function of the $r_i$. This means that the Poincar\'e series is that of a degree-shifted polynomial algebra, as follows.

\begin{lemma}{\cite{Welcher}} \label{WelcherThm} For $k \geq 1,$ \[
\pushQED{\qed} 
P(M_{k-1}/M_k;t) = \frac{t^{2^{k+1}-2}}{\prod_{1 \leq i \leq k}(1-t^{2^{i}-1})}. \qedhere
\popQED
\] 
\end{lemma}

\begin{remark} \label{APoincare}  We prefer to go via $P(M_{k-1}/M_k;t)$ because it leads more cleanly to a result which seems to be just as good for practical purposes, but using precisely the same reasoning as above, one can state an analogous formula for $A(k,n;t)$ directly. For $n \geq 1$ we have $$A(k,n;t) = \sum_{j=1}^k \frac{t^{(n+1)(2^j-1)-2j}}{\prod_{1 \leq i \leq j}(1-t^{2^j-1})}.$$ \end{remark}

We are now ready to prove our main technical result.

\begin{theorem} \label{secretMain} For $n \geq 3$, we have the coefficient-wise inequality $$A(k,n;t) \leq 1 + \sum_{q=1}^\infty ( \sum_{j=1}^k \frac{q^{j-1}}{(j-1)! \cdot 2^{\frac{1}{2}j(j-1)}}) \cdot t^q.$$ \end{theorem}

\begin{remark} \label{APolyBound} Using the formula of Remark \ref{APoincare} instead, to highlight the dependence on $n$, the corresponding result is $$A(k,n;t) \leq 1 + \sum_{q=1}^\infty ( \sum_{j=1}^k \frac{(q-(n-1)(2^j-1)-1-3j)^{j-1}}{(j-1)! \cdot 2^{\frac{1}{2}j(j-1)}}) \cdot t^q.$$ \end{remark}

\begin{proof} By Lemma \ref{toSteenrod}, it suffices to prove the same bound for $P(M_0/M_k;t)$. As vector spaces, $M_0/M_k \cong M_0/M_1 \oplus \dots \oplus M_{k-1}/M_k$. Thus, by Lemma \ref{WelcherThm}, $$P(M_0/M_k;t) = \sum_{j=1}^k \frac{t^{2^{j+1}-2}}{\prod_{1 \leq i \leq j}(1-t^{2^{i}-1})}.$$

It suffices to prove that the $j$-th term of the sum is at most $\frac{q^{j-1}}{(j-1)! \cdot 2^{\frac{1}{2}j(j-1)}} \cdot t^q$. This $j$-th term is the Poincar\'e series for the commutative polynomial algebra $B = \mathbb{F}(x_1, \dots , x_j)$ (where $\lvert x_i \rvert = 2^{i}-1$) shifted by $2^{j+1}-2$.

Since $x_1$ has degree $1$, the dimension of $B_q$ is equal to the cumulative dimension of the algebra obtained by forgetting $x_1$, i.e. to $\bigoplus_{i=0}^q C_i$, where $C = \mathbb{F}(x_2, \dots , x_k)$.

Now, the dimension of $\bigoplus_{i=0}^q C_i$ is equal to the number of non-negative integer solutions $x_2, \dots , x_j$ to the inequality $$ (2^2-1)x_2 + \dots + (2^j-1) x_j \leq q.$$

This is precisely the number of integer points in the closed simplex defined by this inequality and the coordinate hyperplanes. There is a standard upper bound for this quantity (see for example \cite{YauZhang}, which also gives a stronger result - we do not use the stronger result because it only seems to muddy the formula). This gives
\begin{align*}\dim(B_q)  & \leq \frac{1}{{(j-1)}!}(1+\sum_{i=2}^{j} \frac{2^i-1}{q})^{j-1} \prod_{i=2}^j\frac{q}{2^i-1} \\
& \leq \frac{q^{j-1}}{(j-1)!}(1+\frac{1}{q}\sum_{i=2}^{j} 2^i)^{j-1} \prod_{i=2}^j\frac{1}{2^{i-1}} \\
& \leq \frac{1}{(j-1)! \cdot 2^{\frac{1}{2}j(j-1)}}(q+(2^{j+1}-2))^{j-1}.
\end{align*}

The shift applied in the summation is equivalent to replacing $q$ by $q-(2^{j+1}-2)$, which gives the result. \end{proof}

We are now ready to prove Theorem \ref{mainThm}.

\begin{proof}[Proof of Theorem \ref{mainThm}] Lemma \ref{raisonDetre} and Theorem \ref{secretMain} imply that
\begin{align*}\ell_2(\pi_{q+n}P_{2^{k}}S^{n}) & \leq \ellhat(q) \cdot \sum_{i=1}^q A(k,n;t)_i \\
& \leq \ellhat(q) \cdot \sum_{i=1}^q ( \sum_{j=1}^k \frac{i^{j-1}}{(j-1)! \cdot 2^{\frac{1}{2}j(j-1)}})  \\
& \leq \ellhat(q) \cdot \sum_{j=1}^k \int_{x=0}^{q+1} \frac{x^{j-1}}{(j-1)! \cdot 2^{\frac{1}{2}j(j-1)}} \\
& = \ellhat(q) \cdot \sum_{j=1}^k \frac{(q+1)^j}{j! \cdot 2^{\frac{1}{2}j(j-1)}},
\end{align*}
as required. \end{proof}

\begin{remark} Using the formula of Remark \ref{APolyBound} and the fact that $2^j-1 \geq 1$ for $j \geq 1$, one reintroduces some dependency on $n$ to Theorem \ref{mainThm}, and obtains $$\ell_2(\pi_{q+n}P_{2^{k}}S^{n}) \leq \ellhat(q) \cdot \sum_{j=1}^k \frac{(q+2-n)^j}{j! \cdot 2^{\frac{1}{2}j(j-1)}}.$$ \end{remark}

\printbibliography

\end{document}